\documentclass[a4paper,11pt]{article}

\usepackage[utf8]{inputenc}
\usepackage{lmodern}
\usepackage{url}
\usepackage{amsmath, amsthm, amssymb, amscd, accents, bm}
\usepackage{mathtools}
\usepackage{mdwlist}
\usepackage{paralist}
\usepackage{graphicx}
\usepackage{epstopdf}
\usepackage{datetime}
\usepackage{afterpage}
\usepackage[dvipsnames]{xcolor}
\usepackage{hyperref}
\usepackage{authblk}

\newcommand\rurl[1]{%
  \href{https://#1}{\nolinkurl{#1}}%
}

\newtheorem{definition}{Definition}[section]
\newtheorem{theorem}[definition]{Theorem}
\newtheorem{lemma}[definition]{Lemma}
\newtheorem{corollary}[definition]{Corollary}
\newtheorem{proposition}[definition]{Proposition}

\newtheorem{example}[definition]{Example}

\newtheoremstyle{rem}
  {}
  {}
  {}            
  {}            
  {\bfseries}   
  {}            
  { }    
  {}            
\theoremstyle{rem}
\newtheorem{remark}[definition]{Remark}

\def\N{{\mathbb N}}
\def\Z{{\mathbb Z}}
\def\R{{\mathbb R}}
\def\T{{\mathbb T}}
\def\C{{\mathbb C}}
\def\Q{{\mathbb Q}}

\def\1{\textbf{1}}

\newcommand{\norm}[2]{\left\|#1\right\|_{{#2}}}
\newcommand{\lspan}{{\mathrm{span} \,}}

\newcommand{\supp}{{\mathrm{supp}}}

\newcommand{\lt}{{L^2(\R)}}

\newcommand{\ift}{{\mathcal{F}^{-1}}}
\newcommand{\ft}{{\mathcal{F}}}

\usepackage[colorinlistoftodos]{todonotes}

\makeatletter
\newcommand\thankssymb[1]{\textsuperscript{\@fnsymbol{#1}}}

\begin{document}

\title{Injectivity of Gabor phase retrieval from lattice measurements}
\author{Philipp Grohs\thanks{Faculty of Mathematics, University of Vienna, Oskar-Morgenstern-Platz 1, 1090 Wien, Austria, \href{mailto:philipp.grohs@univie.ac.at}{philipp.grohs@univie.ac.at}  \textbf{and} Johann Radon Institute of Applied and Computational Mathematics, Austrian Academy of Sciences, Altenbergstrasse 69, 4040 Linz, Austria, \href{mailto:philipp.grohs@oeaw.ac.at}{philipp.grohs@oeaw.ac.at} \textbf{and} Research Network DataScience@UniVie, University of Vienna, Kolingasse 14-16, 1090 Vienna, Austria} \, and  Lukas Liehr\thanks{Faculty of Mathematics, Universität Wien, Oskar-Morgenstern-Platz 1, 1090 Wien, Austria, \href{mailto:lukas.liehr@univie.ac.at}{lukas.liehr@univie.ac.at}}
}
\date{\today}
\maketitle

\begin{abstract}
We establish novel uniqueness results for the Gabor phase retrieval problem: if $\mathcal{G} : \lt \to L^2(\R^2)$ denotes the Gabor transform then every $f \in L^4[-\tfrac{c}{2},\tfrac{c}{2}]$ is determined up to a global phase by the values $|\mathcal{G}f(x,\omega)|$ where $(x,\omega)$ are points on the lattice $b^{-1}\Z \times (2c)^{-1}\Z$ and $b>0$ is an arbitrary positive constant. 
This for the first time shows that compactly-supported, complex-valued functions can be uniquely reconstructed from lattice samples of their spectrogram. Moreover, by making use of recent developments related to sampling in shift-invariant spaces by Gröchenig, Romero and Stöckler, we prove analogous uniqueness results for functions in shift-invariant spaces with Gaussian generator.
Generalizations to nonuniform sampling are also presented. Finally, we compare our results to the situation where the considered signals are assumed to be real-valued.

\vspace{1em}

\end{abstract}

\vspace{1.8cm}

\break

\section{Introduction}

The phase retrieval problem concerns the reconstruction of a given signal $f\in \mathcal{C} $ in a signal class $\mathcal{C}$ contained in a given Banach space $\mathcal{B}$ from phaseless linear measurements. These measurements are of the form $\{|\Phi_\lambda(f)|: \lambda \in X\}$ with $(\Phi_\lambda)_{\lambda\in X}\subseteq  \mathcal{B}'$ constituting a family of linear functionals on $\mathcal{B}$. Problems of this kind arise in a large number of applications, most notably in coherent diffraction imaging and audio processing where the measurements $(\Phi_\lambda(f))_{\lambda\in X}\subseteq  \C$ typically arise as (samples of) the Fourier- or Gabor transform of a signal $f$, see for example the recent surveys \cite{Strohmer, GrohsKoppensteinerRathmair} and the references therein. The question of when any given function $f\in \mathcal{C}$ is uniquely determined from its phaseless linear measurements \emph{up to a global phase}, meaning that $|\Phi_\lambda(f)|=|\Phi_\lambda(h)|$ for all $\lambda \in X$ and some $f,h\in \mathcal{C}$ implies that there exists a global phase factor $\tau \in \T\coloneqq \{ z \in \C : |z| = 1 \}$ with $f=\tau h$, in general constitutes a difficult mathematical problem, especially when $\mathcal{B}$ is infinite-dimensional and $X$ is discrete as it would be the case in most practical applications. In the present paper, we study the phase retrieval problem arising from measurements of the Gabor transform. The Gabor transform $\mathcal{G} : \lt \to L^2(\R^2)$ is defined by
$$
\mathcal{G}f(x,\omega) = \int_\R f(t)\varphi(t-x)e^{-2\pi i t \omega} \, dt,
$$
where $\varphi$ is the Gaussian window function $\varphi(t) = e^{-\pi t^2}$. If not specified otherwise, all functions $f \in \lt$ are assumed to be complex-valued. For a subset $X \subseteq   \R^2$ of the time-frequency plane we are interested in reconstructing $f\in \mathcal{C} \subseteq   L^2(\R)$ from 
$
\{ |\mathcal{G}f(x,\omega)|  :  (x,\omega) \in X \}
$,
the corresponding spectrogram measurements of the Gabor transform given on $X$. Such problems arise, for instance, in ptychography \cite{da2015elementary} and audio processing \cite{pruuvsa2017phase}. In the language above, the linear functionals $\Phi_\lambda$ are the maps $\Phi_\lambda(f) \coloneqq \langle f,T_xM_\omega \varphi \rangle$ with $\lambda=(x,\omega) \in \R^2$, $\langle \cdot ,\cdot \rangle$ the $L^2$-inner product and $T_x, M_\omega$ the translation and modulation operator, defined by $T_x\psi(t) = \psi(\cdot - x)$ and $M_\omega \psi(t) = e^{2\pi i \omega t}\psi(t)$, respectively. In order to work with a concise notion of uniqueness in the absence of phase information we introduce the following condition.

\begin{definition}
Let $\mathcal{C} \subseteq \lt$ and $X \subseteq \R^2$. We say that the pair $(\mathcal{C},X)$ satisfies condition \textbf{(U)} if every $f \in \mathcal{C}$ is determined up to a global phase by the values $\{ |\mathcal G f(z)| : z \in X \}$. In formulas,
\[
\forall f,h \in \mathcal{C} : ( \, |\mathcal{G}f(z)| = |\mathcal{G}h(z)| \ \forall z \in X \implies \exists \tau \in \T : f=\tau h \, ) \tag{\textbf{U}}.
\]
\end{definition}

\subsection{Previous work}\label{sec:previous_work}

In the following, we provide a list of pairs $(\mathcal{C},X)$ which are known to satisfy condition $(\textbf{U})$.
\begin{enumerate} 
\item \label{itm:pair1} It is well-known that every $f \in \lt$ is determined up to a global phase by the measurements $\{ |\mathcal{G}f(x,\omega)| : (x,\omega) \in \R^2 \}$ and therefore $(\lt,\R^2)$ satisfies condition \textbf{(U)}. This is a direct consequence of the ambiguity-relation
$$\ft (|\mathcal{G}f|^2)(x,\xi) = A \hat f(x,\xi) \, \overline{A \varphi(x,\xi)},$$
and the fact that the ambiguity function $A\varphi$ does not vanish on $\R^2$ \cite[Theorem A.2]{GrohsRathmair}.
In the previous identity, the operator $\ft : L^2(\R^n) \to L^2(\R^n)$ denotes the Fourier transform, defined on $L^1(\R^n) \cap L^2(\R^n)$ via
$$
\ft f(\omega) = \hat f(\omega) = \int_{\R^n} f(x)e^{-2\pi i x \cdot \omega} \, dx
$$
and extends to a unitary operator on $L^2(\R^n)$.
In practice, the measurements $|\mathcal{G}f(x,\omega)|$ are never given on whole $\R^2$ and one seeks to prove similar uniqueness results where $X$ is "small" in the sense that $X$ is, for instance, discrete or a set of measure zero. 

\item  \label{itm:pair2} Assume that $e^{i\alpha_1} \R, e^{i\alpha_2}\R \subseteq \C$ are two lines which satisfy the condition $\alpha_1 - \alpha_2 \notin \pi\Q$. A result due to Jaming shows that every entire function of finite order is determined up to a global phase from its modulus on $e^{i\alpha_1} \R \cup e^{i\alpha_2}\R$ \cite[Theorem 3.3]{Jaming}. This result was recently extended by Perez \cite{perez_2021}. The Gabor transform is, except for a non-zero weighting factor and a reflection, an entire function. This follows from the relation
$$
\mathcal{G}f(x,-\omega) = e^{\pi i x \omega} Bf(z)e^{-\frac{\pi}{2}|z|^2}
$$
where $z = x+i\omega$ and $Bf$ is the Bargmann transform of $f$ \cite[Proposition 3.4.1]{Groechenig}. The Bargmann transform of a function $f \in \lt$ is an entire function of order less or equal than two \cite[Theorem 2]{BrianCHall}. As a direct consequence, Jaming's theorem yields the following choice of $\mathcal{C}$ and $X$.

\begin{theorem}\label{thm:jaming}
Let $\alpha_1,\alpha_2 \in [0,2\pi)$ with $\alpha_1 - \alpha_2 \notin \pi\Q$. Set $\mathcal{C} = \lt$ and $X = e^{i\alpha_1} \R \cup e^{i\alpha_2}\R$. Then $(\mathcal{C},X)$ satisfies condition \textbf{(U)}.
\end{theorem}

Jaming's result is based on Hadamard's factorization theorem and its proof demonstrates that the two lines $e^{i\alpha_1} \R, e^{i\alpha_2}\R$ can be replaced by sets $A \subseteq e^{i\alpha_1} \R, B \subseteq e^{i\alpha_2} \R$ where $A,B$ are sets of uniqueness of entire functions of finite order (e.g. $A$ and $B$ contain a limit point) \cite[Remark 3.4]{Jaming}. In this case, $(\lt, A \cup B)$ satisfies condition \textbf{(U)} as well.

\item \label{itm:pair3}
It follows readily from Theorem \ref{thm:jaming} that $(\lt, \mathcal{O})$ satisfies condition \textbf{(U)} whenever $\mathcal{O}$ is an open subset of $\R^2$. For if $\mathcal{O} \subseteq \R^2 \simeq \C$ is open then we can find two lines $e^{i\alpha_1} \R, e^{i\alpha_2}\R \subseteq \C$ satisfying $\alpha_1 - \alpha_2 \notin \pi\Q$ such that $A = e^{i\alpha_1} \R \cap \mathcal{O}$ and $B = e^{i\alpha_2} \R \cap \mathcal{O}$ contain a limit point.

\item \label{itm:pair4} An example where the sampling set $X$ is fully discrete can be found in \cite{AlaifariWellershoff}. In this article the authors show that if $\mathcal{C} = PW_b^2(\R,\R)$ is the class of real-valued Paley-Wiener functions,
\begin{equation}\label{pw}
PW_b^2(\R, \R) = \{ f \in L^2(\R,\R) : \supp(\hat{f}) \subseteq [-\tfrac{b}{2}, \tfrac{b}{2}] \},
\end{equation}
then the choice $X = (2b)^{-1}\Z \times \{ 0 \}$ implies that $(\mathcal{C}, X)$ satisfies condition \textbf{(U)} \cite[Theorem 2.5]{AlaifariWellershoff}. We will return to this example in Section \ref{sec:main} when we prove uniqueness results for compactly-supported, complex-valued signals.
\end{enumerate}

To summarize, \ref{itm:pair1}, \ref{itm:pair2} and $\ref{itm:pair3}$ above consider non-discrete sampling sets. In practice, however, the spectrogram can only be sampled on a discrete set $X$. 
The only result applying to this case is \ref{itm:pair4} which is, however, only valid for the restricted class of real-valued and band-limited signals, whereas applications in diffraction imaging would rather require results for more general function classes. 
In particular, it is unknown which classes of complex-valued functions (if any) can be uniquely recovered from discrete samples $X \subseteq \R^2$ of their spectrograms. In the present article we are primarily interested in the setting where $X$ is a lattice in the time-frequency plane, i.e. $X=A\Z^2$ with $A \in \mathrm{GL}_2(\R)$ an invertible matrix. At this juncture, we emphasize that if $X$ is a lattice then a prior restriction of the signal class $\mathcal{C}$ to a proper subspace of $\lt$ is necessary. This follows from a recent result which states that uniqueness in $\mathcal{C} = \lt$ is never achieved if $X$ is a lattice, even if one has the freedom of changing the window function $\varphi$ of the short-time Fourier transform \cite[Theorem 1.2]{grohsLiehr3}.

\begin{theorem}
Let $0 \neq g \in \lt$ be a window function and $$V_gf(x,\omega) = \int_\R f(t)\overline{g(t-x)}e^{-2\pi i \omega t} \, dt$$ the short-time Fourier transform of $f$ with respect to the window $g$. For every lattice $X = A\Z^2, A \in \mathrm{GL}_2(\R)$, there exists $f_1,f_2 \in \lt$ such that
\begin{enumerate}
    \item $|V_gf_1(z)| = |V_gf_2(z)|$ for every $z \in X$ and
    \item $f_1$ and $f_2$ do not agree up to a global phase.
\end{enumerate}
\end{theorem}

\subsection{Contributions}
We significantly improve on previous work by adding two more realistic examples of pairs $(\mathcal{C},X)$ which satisfy condition $(\textbf{U})$. The first result concerns compactly-supported functions. This is a natural and common signal assumption made in a variety of imaging applications such as ptychography \cite{111,222,333,444}. The goal in these applications is to reconstruct a finite object from spectrogram samples. The object is modeled as a bounded function having compact support. The following uniqueness result from lattice samples holds in this case.

\begin{theorem}
Let $\mathcal{C}=L^4[-\tfrac{c}{2},\tfrac{c}{2}]$ and let $b>0$ be an arbitrary positive constant. If $$X=\frac{1}{b}\Z\times \frac{1}{2c}\Z$$ then $(\mathcal{C},X)$ satisfies condition $(\mathbf{U})$.
\end{theorem}
\begin{proof}
See Theorem \ref{generalization}.
\end{proof}

The proof of the previous theorem uses Shannon's sampling theorem to sample in frequency space as well as Zalik's theorem (Theorem \ref{thm:zalik}) on the completeness of discrete Gaussian translates. Our second main result, Theorem \ref{main_theorem2}, provides a uniqueness theorem for Gaussian shift-invariant spaces $V_\beta^1(\varphi)$,
$$
V_\beta^1(\varphi) \coloneqq \left \{ f \in L^1(\R) : f = \sum_{k \in \Z} c_k \varphi(\cdot - \beta k), \ c \in \ell^1(\Z) \right \}.
$$

\begin{theorem}
Let $\beta \in \R_{>0}\setminus \Q_{>0}$ and $\mathcal{C} = V_\beta^1(\varphi)$. Further, let $\varepsilon>0$ be an arbitrary positive constant. If
$$
X=\frac{\beta}{2+\varepsilon} \Z\times \Z
$$
then $(\mathcal{C},X)$ satisfies condition $(\mathbf{U})$.
\end{theorem}
\begin{proof}
See Theorem \ref{main_theorem2}.
\end{proof}

The proof of the previous theorem makes use of recent uniqueness results for shift-invariant spaces \cite{GroechenigRomeroStoeckler} as well as Weyl's equidistribution theorem. Observe that in this result the elements of the signal class $\mathcal{C}$ are neither compactly-supported nor band-limited or real-valued which thus represents a new level of generality compared to previous results. In addition, a uniqueness result which holds for arbitrary $\beta>0$ will be presented. We further show that the obtained uniqueness results for shift-invariant spaces are, to a certain degree, sharp. Moreover, we generalized to non-uniform sampling as well as phase-space rotations (i.e. fractional Fourier transforms) of the considered signal classes. 
Finally, all results will be compared in detail to the setting where the considered signals (resp. their Fourier transform) are assumed to be real-valued. The striking observation is that sampling on a 1-dimensional lattice suffices to obtain uniqueness provided that a real-valuedness property holds. If the 1-dimensional lattice gets extended to a 2-dimensional lattice then uniqueness is achieved for general complex-valued maps. Formally, we have the following statement.

\begin{theorem}
Let $\varepsilon,\beta,b,c>0$ and define the four spaces
\begin{equation*}
    \begin{split}
        \mathcal{C}_\R^1 & = \{ f \in L^4[-\tfrac{c}{2},\tfrac{c}{2}] : \ft f \ \text{is real-valued} \, \}, \\
        \mathcal{C}_\C^1 & = L^4[-\tfrac{c}{2},\tfrac{c}{2}], \\
        \mathcal{C}_\R^2 & = \{ f \in V_\beta^1(\varphi) : f \ \text{is real-valued} \, \}, \\
        \mathcal{C}_\C^2 & = V_\beta^1(\varphi).
    \end{split}
\end{equation*}
Then the following holds:
\begin{enumerate}
    \item if $X_\R^1 = \{ 0 \} \times \tfrac{1}{2c}\Z$ then $(\mathcal{C}_\R^1, X_\R^1)$ satisfies condition \textbf{(U)},
    \item if the 1-dimensional lattice $X_\R^1$ gets extended to the 2-dimensional lattice $X^1_\C = \tfrac{1}{b}\Z \times \tfrac{1}{2c}\Z$ then $(\mathcal{C}_\C^1, X_\C^1)$ satisfies condition \textbf{(U)},
    \item if $X_\R^2 = \tfrac{2}{\beta + \varepsilon} \Z \times \{ 0 \}$ then $(\mathcal{C}_\R^2, X_\R^2)$ satisfies condition \textbf{(U)},
    \item if the 1-dimensional lattice $X_\R^2$ gets extended to the 2-dimensional lattice $X^2_\C = \tfrac{2}{\beta + \varepsilon}\Z \times \Z$ then $(\mathcal{C}_\C^2, X_\C^2)$ satisfies condition \textbf{(U)}.
\end{enumerate}
\end{theorem}
\begin{proof}
Statement 1 and 2 will be shown in Theorem \ref{thm:rc_c_1}. Statement 3 and 4 are the content of Theorem \ref{thm:rc_c_2}.
\end{proof}

\section{Preliminaries}\label{sec:basicresults}
This section collects several preliminary results which will be used throughout the paper. In Section \ref{sec:sampling} we provide uniqueness results for Paley-Wiener and shift-invariant spaces, Section \ref{sec:muentz} discusses density results for translates of functions, and Section \ref{sec:frft} collects basic properties of the fractional Fourier transform.

\subsection{Uniqueness sets}\label{sec:sampling}

Our uniqueness results make use of uniqueness theorems for Paley-Wiener spaces and shift-invariant spaces with Gaussian generator. The notion of a uniqueness set is defined as follows.

\begin{definition}
Let $Q \subseteq C(\R)$ be a class of real- or complex-valued, continuous functions. A set $\Lambda \subseteq \R$ is called uniqueness set for $Q$ if the implication
$$
\left ( f,h \in Q : f(\lambda) = h(\lambda) \  \forall \lambda \in \Lambda \right ) \implies f(t)=h(t) \ \forall t \in \R
$$
is satisfied.
\end{definition}

Sufficient conditions on a set $\Lambda$ to be a uniqueness set for Paley-Wiener spaces and shift-invariant spaces are given by means of their effective Beurling-Malliavin density and lower Beurling density, respectively. In the following we give a brief summary of the results relevant for the upcoming proofs.

\subsubsection{Paley-Wiener spaces}\label{subsubsection:pw}

For $b > 0$ we denote by $PW_b^2(\R)$ the Paley-Wiener space which is defined by
$$
PW_b^2(\R) = \{ f \in L^2(\R) : \supp(\hat f) \subseteq [-\tfrac{b}{2}, \tfrac{b}{2}] \}.
$$
The classical sampling theorem of Shannon, Whittaker and Kotel'nikov implies that every $f \in PW_b^2(\R)$ is determined by its samples at $b^{-1}\Z$ \cite[Theorem 6.13]{Higgins}. In particular, $f$ can be recovered by the cardinal sine series
$$
f(t) = \sum_{n \in \Z} f \left( \frac{n}{b} \right) \mathrm{sinc}(bt-n)
$$
and $\Lambda = b^{-1}\Z$ is a uniqueness set for $PW_b^2(\R)$. Clearly, if $f \in PW_b^2(\R)$ then it follows from the representation
$$
f(x) = \int_{-b/2}^{b/2} \hat f(\omega)e^{2\pi i \omega x} \, d\omega
$$
that a set of real numbers $\Lambda \subseteq \R$ is a uniqueness set for $PW_b^2(\R)$ if and only if the complex exponentials
$$
\mathcal{E}_\Lambda = \{ e^{2\pi i\lambda x} : \lambda \in \Lambda \}
$$
are complete in $L^2[-\tfrac{b}{2},\tfrac{b}{2}]$. Determining assumptions on $\Lambda$ which imply completeness of $\mathcal{E}_\Lambda$ is commonly known as the completeness problem of complex exponentials. In this context, a quantity of interest is the radius of completeness of $\Lambda$ which is defined by
$$
R(\Lambda) = \sup \{ a \geq 0 : \mathcal{E}_\Lambda \text{ is complete in } L^2[-\tfrac{a}{2},\tfrac{a}{2}] \}.
$$
A formula for the radius of completeness in terms of a suitable density was derived by Beurling and Malliavin \cite{beurling1962, beurling1967}. To state the corresponding theorem, we start by defining the so-called effective Beurling-Malliavin density. Several equivalent definitions of this density are known. We follow the approach given in \cite[Chapter 5]{Poltoratski}.
For an interval $I \subseteq \R$ denote by $|I|$ its length and by $\mathrm{dist}(0,I)$ its distance to the origin. A sequence $\{ I_n : n \in \Z \}$ of disjoint intervals on $\R$ is called \emph{long} if
$$
\sum_{n \in \Z} \frac{|I_n|^2}{1+\mathrm{dist}^2(0,I_n)} = \infty.
$$
For a sequence $\Lambda \subseteq \R$ of real numbers, the \emph{effective Beurling-Malliavin density (effective BM density)} is defined as
$$
D^*(\Lambda) \coloneqq \sup \{ d \geq 0 : \exists \text{ long } (I_n)_n \text{ such that } \#(\Lambda \cap I_n) \geq d|I_n| \ \forall n \},
$$
where $\#(\Lambda \cap I_n)$ denotes the number of elements in $\Lambda \cap I_n$.
The Beurling-Malliavin theorem states that $D^*(\Lambda)$ coincides with $R(\Lambda)$ \cite[Theorem 26]{Poltoratski}.

\begin{theorem}[Beurling-Malliavin]\label{sampling_pw}
If $\Lambda$ is a discrete real sequence then $R(\Lambda) = D^*(\Lambda).$ Consequently, if $a < D^*(\Lambda) < b$ for some $0<a<b$ then $\mathcal{E}_\Lambda$ is complete in $L^2[-\tfrac{a}{2},\tfrac{a}{2}]$ and incomplete in $L^2[-\tfrac{b}{2},\tfrac{b}{2}]$.
\end{theorem}

\subsubsection{Shift-invariant spaces}

Recent developments in sampling theory due to Gröchenig, Romero and Stöckler show that assumptions on a suitable density of a sequence can be used to obtain uniqueness and sampling results for shift-invariant spaces with totally positive generators of Gaussian type \cite{GroechenigRomeroStoeckler}. We start by defining shift-invariant spaces. Let $1 \leq p \leq \infty$ and let $\beta > 0$ be a step size. The shift-invariant space $V_\beta^p(\varphi)$, generated by the Gaussian $\varphi(t) = e^{-\pi t^2}$, is defined by
\begin{equation}
V_\beta^p(\varphi) \coloneqq \{ f \in L^p(\R) : f = \sum_{k \in \Z} c_k \varphi(\cdot - \beta k), \ c \in \ell^p(\Z) \}.
\end{equation}
Note that functions in $V_\beta^p(\varphi)$ are continuous and if $1\leq p \leq q \leq \infty$ then one has the inclusion
$
V_\beta^1(\varphi) \subseteq V_\beta^p(\varphi) \subseteq V_\beta^q(\varphi)
$
\cite[Corollary 2.5]{GroechenigSurvey}.
Clearly, shift-invariant spaces can be defined in a more general fashion by replacing the Gaussian generator $\varphi$ with a different generator $g$. In particular, the choice $g(t) = \mathrm{sinc}(bt)$ for some $b>0$ implies that $V_{1/b}^2(g) = PW_b^2(\R)$ and we are in the situation of Section \ref{subsubsection:pw}. For a general introduction to shift-invariant spaces we refer to the survey \cite{ron_2001}. 
In order to obtain uniqueness results for shift-invariant spaces, we introduce the \emph{lower Beurling density}.
If $\Lambda \subseteq \R$ then the lower Beurling density of $\Lambda$ is defined by
$$
D^-(\Lambda) \coloneqq \liminf_{r \to \infty} \left( \inf_{a\in \R} \frac{\#(\Lambda \cap [a,a+r))}{r} \right).
$$
In the situation where the generator is the Gaussian $\varphi$, we have the following uniqueness theorem \cite[Theorem 4.4]{GroechenigRomeroStoeckler}.

\begin{theorem}\label{sampling_si}
If $\Lambda \subseteq \R$ has lower Beurling density $D^-(\Lambda) > \beta^{-1}$ then $\Lambda$ is a uniqueness set for $V_\beta^\infty(\varphi)$.
\end{theorem}

\subsection{Müntz-Szász type approximations}\label{sec:muentz}

We aim at combining the above-stated uniqueness theorems with results concerning the density of discrete translates which are based on the classical Müntz-Szász theorem \cite[Theorem 6.1]{LuxemburgKorevaar}.

\begin{theorem}[Müntz-Szász]
Let $[a,b]$ be a compact interval such that $a>0$ and let $\Lambda \subseteq \R$ be a sequence of distinct real numbers. Let $Q$ be one of the spaces $C[a,b]$ or $L^p[a,b]$ with $1\leq p < \infty$. Then the sequence of powers $\{ p_\lambda : \lambda \in \Lambda \}, p_\lambda(t)=t^\lambda,$ is complete in $Q$ if and only if
$$
\sum_{\lambda \in \Lambda}^{}{}^{'} \frac{1}{|\lambda|} = \infty,
$$
where the symbol $\sum'$ indicates that the sum runs over all terms with nonvanishing denominator.
\end{theorem}

The Müntz-Szász theorem can be used to investigate under which assumptions a function space can be spanned by discrete translates of a single function $\psi$. Recall that for $\lambda \in \R$ the shift-operator $T_\lambda$ is defined by $T_\lambda\psi = \psi(\cdot - \lambda)$.

\begin{definition}
Let $I \subseteq \R$ be an interval and let $\Lambda \subseteq \R$ be a set of real numbers. We say that a function $\psi \in \lt$ is a $\Lambda$-generator for $L^2(I)$ if 
$$
\lspan \{ T_\lambda \psi : \lambda \in \Lambda \}
$$
is dense in $L^2(I)$.
\end{definition}

For instance, Wiener's Tauberian theorem \cite[Theorem 1]{Wiener} implies that every $\psi$ which has an a.e. non-vanishing Fourier transform is an $\R$-generator for $\lt$. It is natural to ask if $\Lambda$ can be chosen to be discrete. It then depends on $\Lambda$ and the topology of $I$ whether or not there exists a function that generates $L^2(I)$. For example, there exists no $\Z$-generator for $\lt$ but an arbitrary perturbation of $\Z$ of the form
$
\Lambda = \{ n + a_n : a_n \neq 0, a_n \to 0 \}
$
admits a $\Lambda$-generator for $\lt$ \cite{Olevskii}. If $\psi$ is a Gaussian then the following result due to Zalik provides a Müntz-type condition on $\Lambda$ such that $\psi$ is a $\Lambda$-generator for $L^2(I)$ whenever $I$ is a compact interval \cite[Theorem 4]{Zalik}.

\begin{theorem}[Zalik]\label{thm:zalik}
Let $I \subseteq \R$ be a compact interval and let $\psi(t)=ae^{-b^2(t-c)^2}$ be a Gaussian, $a,b >0, c \in \R$. Further, let $\Lambda \subseteq \R$ be a sequence of distinct real numbers. Then $\psi$ is a $\Lambda$-generator for $L^2(I)$ if and only if the series $\sum'_{\lambda \in \Lambda} |\lambda|^{-1}$ diverges.
\end{theorem}

\subsection{Fractional Fourier transform}\label{sec:frft}

Let $\{ h_n : n \in \N_0 \} \subseteq \lt$ be the Hermite basis functions for $\lt$,
$$
h_n(t) = \frac{2^{1/4}}{\sqrt{n!}} \left( -\frac{1}{\sqrt{2\pi}} \right)^n e^{\pi t^2} \left( \frac{d}{dt} \right)^n e^{-2\pi t^2}.
$$
For $\theta \in \R$, the operator $\ft_\theta : \lt \to \lt$, defined by
$$
\ft_\theta f = \sum_{n \in \N_0} e^{-i\theta n} \langle f,h_n \rangle h_n,
$$
is called the fractional Fourier transform of order $\theta$. Denoting by  $u_\theta : \R \to \C$ the chirp modulation function $u_\theta(t) = e^{-i\pi|t|^2 \cot \theta}$ then the fractional Fourier transform satisfies the relation
$
\ft_\theta f(\omega) = c_\theta u_\theta(\omega) \ft(u_\theta f)(\tfrac{\omega}{\sin \theta})
$
where $c_\theta = \sqrt{1-i\cot\theta}$ is a normalization factor and $\ft : \lt \to \lt$ denotes the ordinary Fourier transform. If we define, in addition, the rotation matrix $R_\theta \in \R^{2 \times 2}$ by
\begin{equation}\label{eq:rotationmatrix}
    R_\theta \coloneqq \left( {\begin{array}{cc}
   \cos \theta & -\sin \theta \\
   \sin\theta & \cos\theta \\
  \end{array} } \right)
\end{equation}
then the operator $\ft_\theta$ has the following properties \cite{irarrazaval2011,Ozaktas2001}.

\begin{lemma}\label{lma:frac_properties}
Let $\varphi(t) = e^{-\pi t^2}$ be the standard Gaussian and let $\theta, \eta \in \R$. Then the fractional Fourier transform enjoys the following properties.

\begin{enumerate}
\item \label{itm:prop1} $\ft_\theta \varphi = \varphi$ for every $\theta \in \R$ whenever $\varphi$ is the Gaussian $\varphi(t) = e^{-\pi t^2}$
\item \label{itm:prop2} $\ft_\theta\ft_\eta = \ft_{\theta+\eta}$ for every $\theta,\eta \in \R$
\item \label{itm:prop3} $\ft_{\frac{\pi}{2}} = \ft$, $\ft_{-\frac{\pi}{2}} = \ft^{-1}$ and $\ft_0 = \mathrm{Id}$
\item \label{itm:prop4} $A(f,g)(R_\theta z) = A(\ft_\theta f, \ft_\theta g)(z)$ for every $z \in \R^2$, $\theta \in \R$ and every $f,g \in \lt$ where $A(f,g)$ denotes the cross-ambiguity function of $f$ and $g$,
$$
A(f,g)(x,\omega) = \int_\R f(t+\tfrac{x}{2})\overline{g(t-\tfrac{x}{2})}e^{-2\pi i t \omega} \, dt.
$$
\end{enumerate}
\end{lemma}

\section{Main results}\label{sec:main}
In this section, we prove the main results of the present paper. In Section \ref{sec:compsupp} we derive uniqueness results for compactly-supported functions followed by a comparison of our results to the real-valued setting. Section \ref{sec:SI} establishes analogous results for Gaussian shift-invariant spaces.
\subsection{Compactly-supported functions}\label{sec:compsupp}

For a function $h : \R \to \C$ and a value $\tau \in \R$ we introduce the shorthand notation
\begin{equation}\label{shorthand_notation}
h_\tau \coloneqq (T_\tau h)\overline{h},
\end{equation}
where $T_\tau$ is the translation operator.

\begin{theorem}\label{main_theorem}
Let $b,c>0$ and $\mathcal{C} = L^4[-\tfrac{c}{2},\tfrac{c}{2}]$. Then the choice $$X = \tfrac{1}{b}\Z \times \tfrac{1}{2c}\Z$$ implies that the pair $(\mathcal{C},X)$ satisfies condition \textbf{(U)}.
In other words, every $f \in L^4[-\tfrac{c}{2},\tfrac{c}{2}]$ is determined up to a global phase by the measurements
$$
\{ |\mathcal{G}f(x, \omega)| \, : \, (x,\omega) \in \tfrac{1}{b}\Z \times \tfrac{1}{2c}\Z \}.
$$
\end{theorem}
\begin{proof}

We denote by $\varphi(t) = e^{-\pi t^2}$ the Gaussian window function. Further, we define $p_x(t) \coloneqq f(t+\frac{x}{2})\overline{\varphi(t-\frac{x}{2})}$. From the definition of the cross-ambiguity function it follows that
$$
A(f,\varphi)(x,\omega)\overline{A(f,\varphi)(x,\omega)} = \ft p_x(\omega) \overline{\ft p_x (\omega)} = \ift(\ft^2 p_x * \overline{p_x})(\omega).
$$
The convolution $\ft^2p_x * \overline{p_x}$ can be written as
$$
(\ft^2p_x * \overline{p_x})(\omega) = \int_\R T_\omega f(k) \, \overline{f(k)} \, T_x[\overline{T_\omega \varphi(k)} \, \varphi(k)] \, dk = \langle f_\omega, T_x \varphi_\omega \rangle,
$$
where we used the change of variables $k \mapsto k - \frac{x}{2}$ in the first equality as well as the notation as defined in equation \eqref{shorthand_notation}. From the fact that $\mathcal{G}f(x,\omega) = e^{-\pi i x \omega} A(f,\varphi)(x,\omega)$ we obtain the identity
\begin{equation}\label{ft_relation}
\ft_2(|\mathcal{G}f|^2)(x,\omega) = \langle f_\omega, T_x \varphi_\omega \rangle,
\end{equation}
where $\ft_2$ denotes the Fourier transform with respect to the second argument. Observe that $f_\omega = 0$ a.e. whenever $|\omega| > c$. Consequently, for every $x \in \R$ we have $|\mathcal{G}f(x,\cdot)|^2 \in PW_{2c}^2(\R)$. Now suppose that a function $h \in L^4[-\tfrac{c}{2},\tfrac{c}{2}]$ satisfies
$$
|\mathcal{G}f(z)| = |\mathcal{G}h(z)|
$$
for every $z \in \frac{1}{b}\Z \times \frac{1}{2c}\Z$. By Shannon's sampling theorem we know that $\frac{1}{2c}\Z$ is a uniqueness set for $PW_{2c}^2(\R)$. Consequently, the functions $|\mathcal{G}f(\lambda,\cdot)|^2$ and $|\mathcal{G}h(\lambda,\cdot)|^2$ agree for ever $\lambda \in \frac{1}{b}\Z$. Taking the Fourier transform with respect to the second argument and consulting equation \eqref{ft_relation} gives the relation
$$
\langle f_\omega - h_\omega, T_\lambda\varphi_\omega \rangle = 0
$$
which holds for every $\omega \in \R$ and every $\lambda \in \frac{1}{b}\Z$. Observe, that for a fixed $\omega$, the function $\varphi_\omega$ is a Gaussian of the form $\varphi_\omega(t) = ae^{-b^2(t-d)^2}$ for some $a,b>0, d \in \R$. Clearly, the series 
$$
\sum_{\lambda \in \frac{1}{b}\Z}^{}{}^{'} \frac{1}{|\lambda|} = b \sum_{n \in \Z \setminus \{ 0 \}} |n|^{-1}
$$
is divergent and it follows from Zalik's theorem that the system of translates $\{ T_\lambda\varphi_\omega : \lambda \in \frac{1}{b}\Z \}$ is complete in $L^2[-\tfrac{c}{2},\tfrac{c}{2}]$. Note that $f_\omega$ and $h_\omega$ have support in $[-\tfrac{c}{2},\tfrac{c}{2}]$ and by Hölder's inequality $f_\omega - h_\omega \in L^2[-\tfrac{c}{2},\tfrac{c}{2}]$. Therefore, $f_\omega = h_\omega$ for every $\omega \in \R$. The choice $\omega = 0$ shows that
\begin{equation}\label{tau}
|f|^2 = |h|^2
\end{equation}
and $h=0$ whenever $f=0$. If $f \neq 0$, select a point $p \in \R$ for which $f(p) \neq 0$. Then $f(t)\overline{f(p)} = h(t)\overline{h(p)}$ for every $t$. Hence, there exists a $\tau \in \C$ such that $f = \tau h$. Equation \eqref{tau} implies $\tau \in \T$.
\end{proof}

In Theorem \ref{main_theorem} the sampling set $X=\frac{1}{b}\Z \times \frac{1}{2c}\Z$ is a lattice, $X = A\Z^2$ with $A = \mathrm{diag}(\frac{1}{b},\frac{1}{2c})$. This lattice was chosen in such a way that both Shannon's sampling theorem and Zalik's theorem are applicable. In view of Theorem \ref{sampling_pw} and the assumptions of Zalik's theorem, the lattice $A$ can be replaced by more general sampling sets as the following result shows. 

\begin{theorem}\label{generalization}
Suppose that $A,B \subseteq \R$ are two sequences of real numbers with the following properties:
\begin{enumerate}
\item the elements of $A$ are distinct and the series of reciprocals $\sum'_{a \in A} |a|^{-1}$ diverges,
\item $B$ is a uniqueness set for $PW_{2c}^2(\R)$ for some $c>0$.
\end{enumerate}
Then every $f \in L^4[-\tfrac{c}{2},\tfrac{c}{2}]$ is determined up to a global phase by the values $|\mathcal{G}f(A \times B)|$. In particular, the pair $(L^4[-\tfrac{c}{2},\tfrac{c}{2}], A \times B)$ satisfies condition \textbf{(U)}.
\end{theorem}
\begin{proof}
Arguing in a similar fashion as in the proof of Theorem \ref{main_theorem}, we observe that $|\mathcal{G}f(x,\cdot)|^2, |\mathcal{G}h(x,\cdot)|^2 \in PW_{2c}^2(\R)$ for every $x \in \R$ provided that $f,h \in L^4[\frac{c}{2},\frac{c}{2}]$. The assumption on $B$ being a uniqueness set for $PW_{2c}^2(\R)$ shows that for every $a \in A$ the implication
\begin{equation*}
    \begin{split}
         & |\mathcal{G}f(a,b)|^2 = |\mathcal{G}h(a,b)|^2  \ \ \forall b \in B \\
         \implies & |\mathcal{G}f(a,\omega)|^2 = |\mathcal{G}h(a,\omega)|^2 \ \ \ \ \forall \omega \in \R
    \end{split}
\end{equation*}
holds true. Using identity \eqref{ft_relation} and applying the Fourier transform results in the equation
$$
\langle f_\omega , T_{a} \varphi_\omega \rangle = \langle h_\omega , T_{a} \varphi_\omega \rangle
$$
which holds for every $a \in A$ and every $\omega \in \R$. Since $\varphi_\omega$ is a Gaussian, $f_\omega - h_\omega \in L^2[-\frac{c}{2},\frac{c}{2}]$ and $\sum'_{a \in A} |a|^{-1} = \infty$, Zalik's theorem implies that $f_\omega = h_\omega$ for every $\omega \in \R$. This gives $f = \tau h$ for some $\tau \in \T$ which concludes the proof of the statement.
\end{proof}

Note that in Theorem \ref{generalization} the assumption on $B$ being a uniqueness set for $PW_{2c}^2(\R)$ is satisfied if one requires that $B$ has effective Beurling-Malliavin density $D^*(B)>2c$. 
For instance, we could choose $B=b+\beta \Z$ where $b\geq 0$ is a non-negative real number and $0<\beta\leq \frac{1}{2c}$. Further, if $A$ takes the form $A=a+\alpha \Z$ with $a \geq 0$ and $\alpha>0$ then $A$ satisfies condition 1 of Theorem \ref{generalization} which implies that the pair $(L^4[-\tfrac{c}{2},\tfrac{c}{2}], (a+\alpha\Z) \times (b+\beta \Z))$ satisfies condition \textbf{(U)}. This sampling set is a shifted lattice and is visualized in Figure \ref{fig:grd}.

\begin{figure}\label{fig:grd}
  \hspace*{0.4cm}  
  \includegraphics[width=11.5cm]{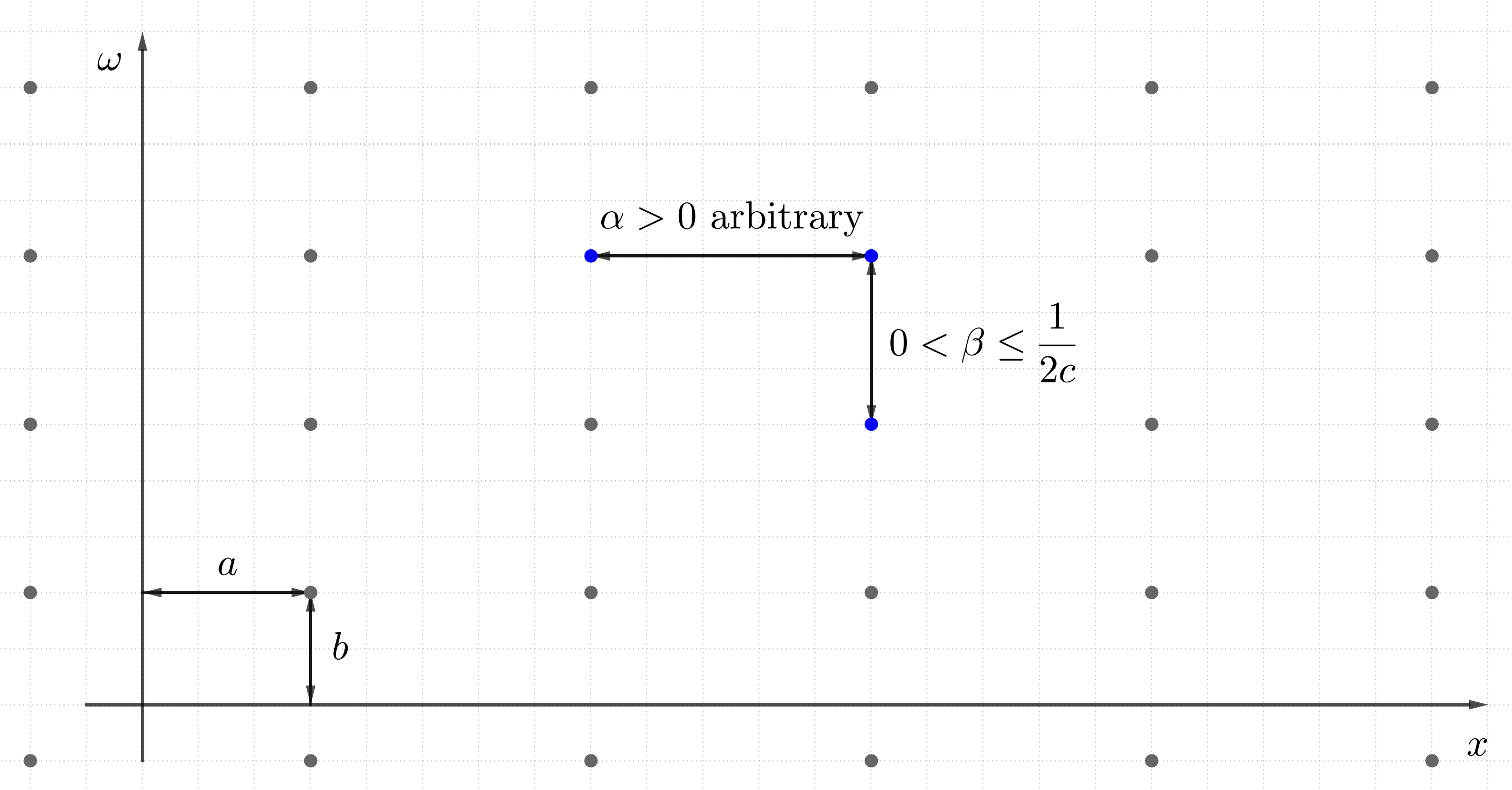}
  \caption{Lattice in the time-frequency plane. Functions in $L^4[-\tfrac{c}{2},\tfrac{c}{2}]$ are determined from measurements given on lattices of the form $(a+\alpha\Z) \times (b+\beta \Z)$ where $a,b \geq 0$ and $\alpha >0$ are arbitrary and $\beta$ satisfies $0<\beta \leq \frac{1}{2c}$.}
\end{figure}
Observe further that in Theorem \ref{generalization} the uniqueness set $B$ does not need to be a priori fixed. In fact, for every single $a \in A$ we can choose a different uniqueness set $S^a \subseteq \R$.

\begin{theorem}
Let $c>0$ and let $A \subseteq \R$ be a sequence of distinct real numbers with the property that the series of reciprocals $\sum'_{a \in A} |a|^{-1}$ diverges. For every $a \in A$ let $S^a$ be a uniqueness set for $PW_{2c}^2(\R)$. Set
$$
X = \bigcup_{a \in A} \{ (a,s) :s \in S^a \}
$$
and $\mathcal{C} = L^4[-\tfrac{c}{2},\tfrac{c}{2}]$. Then $(\mathcal{C},X)$ satisfies condition \textbf{(U)}.
\end{theorem}
\begin{proof}
Suppose that $f,h \in L^4[-\frac{c}{2},\frac{c}{2}]$ are such that $$
|\mathcal{G}f(z)| = |\mathcal{G}h(z)| \ \ \forall z \in X.
$$
With an analogous argument as in the proof of Theorem \ref{generalization} we infer from the assumption on $S^a$ being a uniqueness set for $PW_{2c}^2(\R)$ that
$$
|\mathcal{G}f(a,\omega)| = |\mathcal{G}h(a,\omega)| \ \ \forall a \in A \ \forall \omega \in \R.
$$
An application of the Fourier transform together with the identity \eqref{ft_relation} shows that 
$$
\langle f_\omega - h_\omega , T_{a} \varphi_\omega \rangle = 0 \ \ \forall a \in A \ \forall \omega \in \R.
$$
By assumption, the series of reciprocals $\sum'_{a \in A} |a|^{-1}$ diverges. Therefore, Zalik's theorem implies that $f_\omega = h_\omega$ for every $\omega \in \R$ which in turn yields $f=\tau h$ for some $\tau \in \T$.
\end{proof}

Consulting Section \ref{sec:frft}, new uniqueness results can be obtained via a phase-space rotation. In particular, all of the previous theorems can be stated in an analogous manner for band-limited functions by means of a rotation of the lattice by 90 degrees.

\begin{proposition}
Let $\theta \in \R$ and let $R_\theta$ be the rotation matrix as defined in equation \ref{eq:rotationmatrix}. Further, let $b>0$ and $\mathcal{C}_\theta \coloneqq \ft_{-\theta}L^4[-\tfrac{c}{2},\tfrac{c}{2}]$. If
$$
X_\theta \coloneqq R_\theta( \tfrac{1}{b} \Z \times \tfrac{1}{2c}\Z) \coloneqq \{ R_\theta z : z \in \tfrac{1}{b} \Z \times \tfrac{1}{2c}\Z \}
$$
then $(\mathcal{C}_\theta , X_\theta)$ satisfies condition $\textbf{(U)}$. In particular, all band-limited functions in the space
$$
\mathcal{C}_{\frac{\pi}{2}} = \ift L^4[-\tfrac{c}{2},\tfrac{c}{2}] =  \{ f \in PW_c^2(\R) : \hat f \in L^4[-\tfrac{c}{2},\tfrac{c}{2}] \}
$$
are determined up to a global phase by the measurements
$$
\{ |\mathcal{G}f(z)| : z \in \tfrac{1}{2c}\Z \times \tfrac{1}{b}\Z \}.
$$
\end{proposition}
\begin{proof}
Suppose that $f,h \in \mathcal{C}_\theta$ satisfy
$$
|\mathcal{G}f(R_\theta z)| = |\mathcal{G}h(R_\theta z)|
$$
for every $z \in \frac{1}{2c}\Z \times \frac{1}{b}\Z$. From the properties of the fractional Fourier transform (see Lemma \ref{lma:frac_properties}) it follows that
$
|\mathcal{G}(\ft_\theta f)(z)| = |\mathcal{G}(\ft_\theta h)(z)|
$
for every $z \in \frac{1}{b}\Z \times \frac{1}{2c}\Z$. By assumption, $\ft_\theta f$ and $\ft_\theta h$ are elements of $L^4[-\tfrac{c}{2},\tfrac{c}{2}]$. Hence, the first part of the statement follows from Theorem \ref{main_theorem}. Since $\ft_{-\frac{\pi}{2}} = \ift$ and
$$
R_{\frac{\pi}{2}} = \left( {\begin{array}{cc}
   0 & -1 \\
   1 & 0 \\
  \end{array} } \right)
$$
the second part of the statement is a special case of the first.
\end{proof}

\subsection{Real-valued maps with compact support}\label{subse:rv}

In the previous section, we derived uniqueness theorems for complex-valued functions under a support condition. We wish to elaborate on these results by providing a comparison to the setting where besides a support condition an additional real-valuedness property holds. To that end, suppose that $f \in L^2[-\frac{c}{2},\frac{c}{2}]$ has a real-valued Fourier transform.
Since the Gabor transform satisfies the relation $|\mathcal{G}f(x,\omega)| = |\mathcal{G}\hat f(-\omega,x)|$ whenever $f \in \lt$ and $(x,\omega) \in \R^2$, the uniqueness problem for compactly-supported functions with a real-valued Fourier transform is equivalent to the uniqueness problem in $PW_c^2(\R,\R)$, the class of real-valued Paley-Wiener functions. This is precisely the setting mentioned in Section \ref{sec:previous_work} and considered in \cite{AlaifariWellershoff} where it was shown that if $\mathcal{C} = PW_c^2(\R,\R)$ and $X = \frac{1}{2c}\Z \times \{ 0 \}$ then $(\mathcal{C},X)$ satisfies condition \textbf{(U)}. Note that this result is an application of the fact that a function in $PW_c^2(\R,\R)$ is determined up to a global phase by samples of its modulus \cite{Thakur2011}.

\begin{theorem}[Thakur]\label{thm:thakur}
Let $p,q \in PW_c^2(\R,\R)$ such that $|p(x)|=|q(x)|$ for every $x \in \frac{1}{2c}\Z$. Then there exists a constant $\tau \in \{ 1,-1 \}$ such that $p=\tau q$.
\end{theorem}

The corresponding uniqueness result in $PW_c^2(\R,\R)$ from spectrogram samples follows readily: since $|\mathcal{G}f(x,0)| = |f * \varphi(x)|$, the assumption that
$$
|\mathcal{G}f(z)| = |\mathcal{G}h(z)| \ \forall z \in \tfrac{1}{2c}\Z \times \{ 0 \}
$$
for some $f,h \in PW_c^2(\R,\R)$ is equivalent to the relation
$$
|f * \varphi(x)| = |h * \varphi(x)| \ \forall x \in \tfrac{1}{2c}\Z.
$$
Since $f * \varphi, h * \varphi \in PW_c^2(\R,\R)$, an application of Thakur's theorem and the fact that $\hat \varphi$ does not vanish yields the existence of a $\tau \in \{ 1,-1 \} \subseteq \T$ so that $f=\tau h$. In contrast, if $f$ is complex-valued then $f * \varphi$ is a complex-valued band-limited function in $PW_c^2(\R)$. But functions in $PW_c^2(\R)$ are \emph{not} determined up to a global phase by samples of their moduli, even if the samples are given on the entire real line.

\begin{example}
Suppose that the Fourier transform of $f \in PW_c^2(\R)$ has support in the interval $[-\frac{c}{4},\frac{c}{4}]$. For $\varepsilon \in [-\frac{c}{4},\frac{c}{4}] \setminus \{ 0 \}$ define
$$
f_\varepsilon(t) = e^{2 \pi i \varepsilon t}f(t).
$$
Then the Fourier transform of $f_\varepsilon$ has support in $[-\frac{c}{2},\frac{c}{2}]$ which gives $f_\varepsilon \in PW_c^2(\R)$. Since $\varepsilon \neq 0$ the support of $\ft f$ and $\ft(f_\varepsilon)$ do not coincide. In particular, there exists no $\tau \in \T$ so that $f = \tau f_\varepsilon$. On the other hand, $|f(t)|=|f_\varepsilon(t)|$ for every $t \in \R$.
\end{example}

The fact that complex-valued band-limited signals are not determined by their moduli produces a natural obstacle in generalizing the proof above which is based on Thakur's theorem from the real-valued setting to the complex-valued setting. However, it yields the new insight that Gabor phase retrieval under a real-valuedness assumption is possible via sampling on a 1-dimensional lattice whereas uniqueness for complex-valued maps is achieved by an extension to a 2-dimensional lattice in the time-frequency plane. This is summarized in

\begin{theorem}\label{thm:rc_c_1}
Let $b,c>0$ and consider the class of functions in $L^4[-\tfrac{c}{2},\tfrac{c}{2}]$ which have a real-valued Fourier transform,
$$
\mathcal{C}_\R = \{ f \in  L^4[-\tfrac{c}{2},\tfrac{c}{2}] : \ft f \ \text{is real-valued} \, \}.
$$
If $X_\R = \{ 0 \} \times \frac{1}{2c}\Z$ then $(\mathcal{C}_\R, X_\R)$ satisfies condition \textbf{(U)}. On the other hand, if
$$
\mathcal{C}_\C = L^4[-\tfrac{c}{2},\tfrac{c}{2}]
$$
and $X_\C$ is the extended lattice $X_\C = \frac{1}{b}\Z \times \frac{1}{2c}\Z$ then $(\mathcal{C}_\C,X_\C)$ satisfies condition \textbf{(U)}.
\end{theorem}
\begin{proof}
Suppose that $f,h \in \mathcal{C}_\R$ are such that
\begin{equation}\label{33}
    |\mathcal{G}f(z)| = |\mathcal{G}h(z)|
\end{equation}
for every $z \in \{ 0 \} \times \frac{1}{2c}\Z$. Since $|\mathcal{G}u(x,\omega)| = |\mathcal{G}\hat u (-\omega,x)|$ for every $(x,\omega) \in \R^2$ and every $u \in \lt$ it follows that equation \eqref{33} is equivalent to the validity of the identity
$$
|\mathcal{G}\hat f(x,0)| = |\mathcal{G}\hat h (x,0)|
$$
for every $x \in \frac{1}{2c}\Z$. The inclusion $L^4[-\frac{c}{2},\frac{c}{2}] \subseteq L^2[-\frac{c}{2},\frac{c}{2}]$ implies that $\hat f, \hat h \in PW_c^2(\R ,\R)$. Further, we have
$$
|\mathcal{G}\hat f (x,0)| = |\hat f * \varphi(x)|
$$
with $\hat f * \varphi \in PW_c^2(\R,\R)$ and the same holds true when $\hat f$ gets replaced by $\hat h$. Hence, the modulus of the two functions $\hat f * \varphi, \hat h * \varphi \in PW_c^2(\R,\R)$ agrees on $\frac{1}{2c}\Z$. Thakur's theorem (Theorem \ref{thm:thakur}) yields the existence of a constant $\tau \in \{ 1,-1\} \subseteq \T$ so that $\hat f * \varphi = \tau (\hat h * \varphi)$. By Fourier uniqueness and the fact that $\varphi$ does not vanish, the previous identity implies that $f$ and $h$ agree up to a global phase. The second statement of the theorem is the content of Theorem \ref{main_theorem}.
\end{proof}

\subsection{Shift-invariant spaces with Gaussian generator}\label{sec:SI}

All of the theorems in the previous chapter used the fact that our signal space $\mathcal{C}$ consists of functions which are compactly-supported in the time domain or the frequency domain (or more general: their fractional Fourier transform has compact support). We wish to make use of the uniqueness Theorem \ref{sampling_si} which holds for shift-invariant spaces with Gaussian generator to determine sets $X \subseteq \R^2$ consisting of separated points such that $(V_\beta^1(\varphi), X)$ satisfies condition \textbf{(U)}. Recall that a set $X \subseteq \R^n$ is called separated if
$$
\inf_{\substack{x,x' \in X \\ x \neq x'}} |x-x'| >0
$$
where $|\cdot|$ denotes the Euclidean distance. We first discuss the general case where the considered signals are complex-valued. In analogy to Theorem \ref{thm:rc_c_1} we then proceed by discussing the real-valued case and show that as in the case of compactly-supported functions, real-valued maps in Gaussian shift-invariant spaces are determined by samples on a 1-dimensional lattice whereas uniqueness in the complex case is achieved by an extension to a 2-dimensional lattice. To tackle the complex case we start with an algebraic lemma.

\begin{lemma}\label{lma}
Let $\beta>0$ and $f = \sum_{k \in \Z} c_k \varphi(\cdot - \beta k) \in V_\beta^1(\varphi)$. Then the modulus of the Gabor transform of $f$ can be written as
$$
|\mathcal{G}f(x,\omega)|^2 = \sum_{k \in \Z} \sum_{j \in \Z} c_k \overline{c_j} a(j,k,\beta) M_{\frac{\beta}{2}(j-k)}\varphi(\omega)T_{\frac{\beta}{2}(j+k)}\varphi(x),
$$
where $a(j,k,\beta) = \frac{1}{2}e^{-\frac{\pi\beta^2}{4}(k-j)^2}$.
\end{lemma}
\begin{proof}
If $f = \sum_{k \in \Z} c_k \varphi(\cdot - \beta k)$ then it follows from the properties of the Gabor transform that
$$
\mathcal{G}f(x,\omega) = \sum_{k \in \Z} c_k e^{-2\pi i\beta k\omega} \mathcal{G}\varphi(x-\beta k,\omega).
$$
Note that the Gabor transform of the Gaussian $\varphi$ is given by $\mathcal{G}\varphi(x,\omega) = \frac{1}{\sqrt 2} e^{-\pi ix\omega}e^{-\frac{\pi}{2}(x^2+\omega^2)}$ and therefore
$$
\mathcal{G}f(x,\omega) = \sum_{k \in \Z} \frac{c_k}{\sqrt 2} e^{-2\pi i \beta k\omega} e^{-\pi i(x-\beta k)\omega}e^{-\frac{\pi}{2}((x-\beta k)^2+\omega^2)}.
$$
Consequently, we obtain the identity
$$
|\mathcal{G}f(x,\omega)|^2 = \sum_{k \in \Z} \sum_{j \in \Z} \frac{c_k\overline{c_j}}{2} e^{\pi i \beta (j-k)\omega} e^{-\frac{\pi}{2}((x-\beta k)^2 + (x-\beta j)^2 + 2 \omega^2)}.
$$
But $$e^{-\frac{\pi}{2}((x-\beta k)^2 + (x-\beta j)^2 + 2 \omega^2)} =  e^{-\frac{\pi\beta^2}{4}(k-j)^2} \varphi(\omega) \varphi( x-\tfrac{\beta}{2}(j+k))$$ which yields the statement.
\end{proof}

The previous representation of $|\mathcal{G}f(x,\omega)|^2$ will be used to provide a condition on the step size $\beta$ such that every function in $V_\beta^1(\varphi)$ is determined up to a global phase from lattice measurements. This is the content of

\begin{theorem}\label{main_theorem2}
Let $\beta \in \R_{>0} \setminus \Q_{>0}$ and let $\Lambda \subseteq \R$ be a sequence with lower Beurling density $D^-(\Lambda)>\frac{2}{\beta}$. Then every $f \in  V_\beta^1(\varphi)$ is determined up to a global phase by the measurements
$$
\{ |\mathcal{G}f(x, \omega)| : (x,\omega) \in \Lambda \times \Z \}.
$$
In particular, for every $\varepsilon>0$ and every $\beta \in \R_{>0} \setminus \Q_{>0}$ the pair $(V_\beta^1(\varphi), \frac{\beta}{2+\varepsilon}\Z \times \Z)$ satisfies condition \textbf{(U)}.
\end{theorem}
\begin{proof}
Suppose that $f = \sum_{k \in \Z} c_k \varphi(\cdot - \beta k) \in V_\beta^1(\varphi)$ with $(c_k)_k \in \ell^1(\Z).$ Let $a(j,k,\beta)$ be defined as in Lemma \ref{lma}. Define a sequence $(A_n)_n \subseteq \R$ via
$$
A_n = \sum_{j,k \in \Z, \, j+k =n} c_k\overline{c_j}a(j,k,\beta)M_{\frac{\beta}{2}(j-k)}\varphi(\omega)
$$
where $\omega \in \R$ is fixed. Moreover, for any fixed $x \in \R$ define the sequence $(B_n)_n \subseteq \R$ via
$$
B_n = \sum_{j,k \in \Z, \, j-k =n} c_k\overline{c_j}a(j,k,\beta)T_{\frac{\beta}{2}(j+k)}\varphi(x).
$$
It follows from Lemma \ref{lma} that
$$
|\mathcal{G}f(x,\omega)|^2 = \sum_{n \in \Z} A_n T_{\frac{\beta}{2}n}\varphi(x) = \sum_{n \in \Z} B_n M_{\frac{\beta}{2}n}\varphi(\omega).
$$
Since $|a(j,k,\beta)| \leq 1$ for every $j,k \in \Z$ and every $\beta > 0$ we have $\norm {A} {\ell^1(\Z)} \leq \varphi(\omega) \norm {c} {\ell^1(\Z)}^2$ and $\norm {B} {\ell^1(\Z)} \leq \norm {c} {\ell^1(\Z)}^2$. We conclude that for every $\omega \in \R$, the map $x \mapsto |\mathcal{G}f(x,\omega)|^2$ is an element of the shift-invariant space $V_{\frac{\beta}{2}}^1(\varphi)$. In addition, for every $x \in \R$ the map
$$
\omega \mapsto \frac{|\mathcal{G}f(x,\omega)|^2}{\varphi(\omega)}
$$
is continuous and $\frac{2}{\beta}$-periodic. Now suppose that $h \in V_\beta^1(\varphi)$ satisfies
$$
|\mathcal{G}f(\lambda,n)| = |\mathcal{G}h(\lambda,n)|
$$
for every $(\lambda,n) \in \Lambda \times \Z$. According to the previous arguments, we have $|\mathcal{G}f(\cdot,n)|^2, |\mathcal{G}h(\cdot,n)|^2 \in V_{\frac{\beta}{2}}^1(\varphi)$. By Theorem \ref{sampling_si}, $\Lambda$ is a uniqueness set for $V_{\frac{\beta}{2}}^1(\varphi)$ and therefore
$
|\mathcal{G}f(x,n)| = |\mathcal{G}h(x,n)|
$
for every $x \in \R$ and every $n \in \Z$. Fix $x \in \R$. It follows that
\begin{equation}\label{fractions}
\frac{|\mathcal{G}f(x,\omega)|^2}{\varphi(\omega)} = \frac{|\mathcal{G}h(x,\omega)|^2}{\varphi(\omega)}
\end{equation}
whenever $\omega \in \Z$. By periodicity, the left- and right-hand side of equation \eqref{fractions} agree for every $\omega \in D$ where
$$
D \coloneqq \left \{ n \ \mathrm{mod} \ \frac{2}{\beta} : n \in \Z \right\} = \frac{2}{\beta} \left\{ \frac{\beta}{2}n \ \mathrm{mod} \ 1 : n \in \Z \right\}.
$$
Since $\beta$ is irrational, it follows from Weyl's equidistribution theorem that the sequence $\{ \frac{\beta}{2}n \ \mathrm{mod} \ 1 : n \in \Z \}$ is uniformly distributed in $[0,1]$ (see \cite[Satz 2]{Weyl}). Hence, both sides of equation \eqref{fractions} agree on a dense subset of $[0,\frac{2}{\beta}]$. By periodicity and continuity, we conclude that equation \eqref{fractions} holds for every $\omega \in \R$. Since $x$ was arbitrary it follows that the modulus of the Gabor transform of $f$ agrees with the modulus of the Gabor transform of $h$ on whole $\R^2$. Consequently, there exists a $\tau \in \T$ such that $f = \tau h$ (see Section \ref{sec:previous_work}). The second statement follows from the fact that the lower Beurling density of $\tfrac{\beta}{2+\varepsilon}\Z$ satisfies
$
D^{-}(\tfrac{\beta}{2+\varepsilon}\Z) = \tfrac{2+\varepsilon}{\beta} > \tfrac{2}{\beta}.
$
\end{proof}

In the previous theorem, the step size $\beta$ was chosen in such a way that the sequence $\{ n \ \mathrm{mod} \ \frac{2}{\beta} : n \in \Z \}$ is dense in $[0, \frac{2}{\beta}]$. Clearly, we can drop the assumption on $\beta$ to be irrational and instead work with a sampling set $\Lambda \times B$ where $B=\{ b_n : n \in \Z \} \subseteq \R$ is chosen in such a way that $\{ b_n \ \mathrm{mod} \ \frac{2}{\beta} : n \in \Z \}$ is dense in $[0, \frac{2}{\beta}]$.

\begin{theorem}\label{thm:generalization}
Let $\beta > 0$ and suppose that $\Lambda \subseteq \R$ and $B= \{ b_n : n \in \Z \} \subseteq \R$ are real sequences with the following properties:
\begin{enumerate}
\item the sequence $\Lambda$ has lower Beurling density $D^-(\Lambda)>\frac{2}{\beta}$,
\item the sequence $\{ b_n \ \mathrm{mod} \ \frac{2}{\beta} : n \in \Z \}$ is dense in the interval $[0,\frac{2}{\beta}]$.
\end{enumerate}
Then $(V_\beta^1(\varphi),\Lambda \times B)$ satisfies condition \textbf{(U)}. In particular, both $\Lambda$ and $B$ can be chosen to be separated. In this case, $\Lambda \times B$ constitutes a separated subset of $\R^2$.
\end{theorem}
\begin{proof}
Arguing in a similar fashion as in the proof of Theorem \ref{main_theorem2}, we have $|\mathcal{G}f(\cdot,b_n)|^2, |\mathcal{G}h(\cdot,b_n)|^2 \in V_{\frac{\beta}{2}}^1(\varphi)$ whenever $n \in \Z$ and $f,h \in V_\beta^1(\varphi)$. Since $\Lambda$ has lower Beurling density $D^-(\Lambda)>\frac{2}{\beta}$, it follows from Theorem \ref{sampling_si} that $|\mathcal{G}f(x,b_n)|^2 = |\mathcal{G}h(x,b_n)|^2$ for every $x \in \R$ and every $n \in \Z$. In particular,
\begin{equation}\label{eee}
    \frac{|\mathcal{G}f(x,b_n)|^2}{\varphi(b_n)} = \frac{|\mathcal{G}h(x,b_n)|^2}{\varphi(b_n)}
\end{equation}
for every $x \in \R$ and every $n \in \Z$. Since both the left-hand side and the right-hand side of equation \eqref{eee} are continuous and $\frac{2}{\beta}$-periodic, the density of the set $B$ in the interval $[0,\frac{2}{\beta}]$ implies that $|\mathcal{G}f(x,\omega)|^2 = |\mathcal{G}h(x,\omega)|^2$ for every $(x,\omega) \in \R^2$. Hence, $f$ and $h$ agree up to a global phase (see Section \ref{sec:previous_work}).
\end{proof}

Comparing the previous generalization of Theorem \ref{main_theorem2} with the results in \cite{alaifari2020phase,grohsLiehr3} implies that Theorem \ref{thm:generalization} is sharp with respect to the density condition on the sequence $\{ b_n \ \mathrm{mod} \ \frac{2}{\beta} : n \in \Z \}$.

\begin{remark}[On the sharpness of Theorem \ref{thm:generalization}]\label{remark:sharpness}

Let $\beta > 0$ and $\{ c_k \} \in c_{00}(\Z)$ where $c_{00}(\Z)$ denotes the space of all complex sequences with finitely many non-zero components. Further, let
$$
f = \sum_{k \in \Z} c_k T_{\beta k} \varphi
$$
be a linear combination of $\beta \Z$-shifts of $\varphi$. Clearly, we have $f \in V_\beta^1(\varphi)$. Define a second function $\tilde f \in V_\beta^1(\varphi)$ by
$$
\tilde f = \sum_{k \in \Z} \overline{c_k} T_{\beta k} \varphi,
$$
i.e. $\tilde f$ arises from $f$ via complex conjugation of the sequence $\{ c_k \}$. It was shown in \cite[Theorem 3.1]{grohsLiehr3} that the above choice of $f$ and $\tilde f$ implies that
$$
|\mathcal{G}f(z)| = |\mathcal{G}\tilde f (z)| \ \ \forall z \in \R \times \tfrac{1}{\beta} \Z.
$$
Moreover, it was shown that if the sequence $\{ c_k \}$ is not contained in a line in the complex plane passing through the origin, i.e.
$$
\nexists \alpha \in \R :  \{ c_k \} \subseteq e^{i\alpha} \R,
$$
then $f$ and $\tilde f$ do not agree up to a global phase. Hence, uniqueness in $V_\beta^1(\varphi)$ is not achieved from the sampling set $\Lambda \times B \coloneqq \R \times \tfrac{1}{\beta}\Z$. Note that the set $B = \{ b_n : n \in \Z \} = \frac{1}{\beta}\Z$ with $b_n = \frac{1}{\beta} n$, has the property that
$$
\{ b_n \ \mathrm{mod} \ \tfrac{2}{\beta} : n \in \Z \} = \{ \tfrac{1}{\beta} n \ \mathrm{mod} \ \tfrac{2}{\beta} : n \in \Z \}
$$
is not dense in the interval $[0,\frac{2}{\beta}]$ since it is a set of isolated points. Therefore, condition (2) of Theorem \ref{thm:generalization} is violated. In particular, the previous construction shows that, in general, the uniqueness property does not hold anymore if one drops condition (2) of Theorem \ref{thm:generalization}. Observe further, that every rectangular lattice $a\Z \times \frac{1}{\beta}\Z, a,\beta>0,$ is contained in the set $\R \times \frac{1}{\beta}\Z$. This shows additionally that the pair
$$
(V_\beta^1(\varphi), \tfrac{\beta}{2+\varepsilon} \Z \times \tfrac{1}{\beta} \Z)
$$
does not satisfy condition \textbf{(U)} for every $\varepsilon>0$ and every $\beta>0$. In contrast, Theorem \ref{main_theorem2} states that the pair
$$
(V_\beta^1(\varphi), \tfrac{\beta}{2+\varepsilon} \Z \times \Z)
$$
satisfies condition \textbf{(U)} for every $\varepsilon>0$ and every $\beta \in \R_{>0} \setminus \Q_{>0}$. Hence, uniqueness is achieved by making the sampling rate in frequency direction independent of the step-size $\beta$ of the shift-invariant space $V_\beta^1(\varphi)$.
\end{remark}

Similar to the case of compactly-supported functions, a rotation of the time-frequency plane yields uniqueness results for functions which have the property of being the fractional Fourier transform of a function in $V_\beta^1(\varphi)$. More precisely we have

\begin{proposition}\label{prop:si_frac}
Let $\theta \in \R$, $\beta \in \R_{>0} \setminus \Q_{>0}$ and suppose that $\Lambda \subseteq \R$ has lower Beurling density $D^{-}(\Lambda)>\frac{2}{\beta}$. Then every $f \in \ft_{-\theta}V_\beta^1(\varphi)$ is determined up to a global phase by the measurements
$$
\{ |\mathcal{G}f(z)| : z \in R_\theta(\Lambda \times \Z) \},
$$
where $R_\theta$ denotes the rotation matrix as defined in equation \ref{eq:rotationmatrix}.
\end{proposition}
\begin{proof}
Let $f,h \in \ft_{-\theta}V_\beta^1(\varphi)$, i.e. there exist functions $u,v \in V_\beta^1(\varphi)$ such that $f=\ft_{-\theta}u$ and $h=\ft_{-\theta}v$. Suppose that $z=R_\theta z'$ with $z' \in \Lambda \times \Z$. Then by Lemma \ref{lma:frac_properties} we observe that the equality $|\mathcal{G}f(z)| = |\mathcal{G}h(z)|$ is equivalent to $|\mathcal{G}u(z')| = |\mathcal{G}v(z')|$. Since $z' \in \Lambda \times \Z$ was arbitrary, it follows that the spectrograms of $u$ and $v$ agree on $\Lambda \times \Z$. Therefore, Theorem \ref{main_theorem2} implies the existence of a $\tau \in \T$ such that $u=\tau v$. In particular, $\ft_{-\theta} u = \tau \ft_{-\theta} v$ which yields $f=\tau h$.
\end{proof}

If $f$ is the product of a periodic function and the Gaussian $\varphi$ then a rotation by 90 degrees combined with Proposition \ref{prop:si_frac} implies the following result.

\begin{corollary}
Assume that $f \in \lt$ factors as $f=\varphi u$ where $u \in C^1(\R)$ is $\beta$-periodic. Let $\Lambda$ be a sequence in $\R$ with lower Beurling density $D^-(\Lambda) > 2 \beta$. If $\beta \in \R_{>0} \setminus \Q_{>0}$ then $f$ is determined up to a global phase by the measurements
$$
\{ |\mathcal{G}f(x,\omega)| : (x,\omega) \in \Z \times \Lambda \}.
$$
\end{corollary}
\begin{proof}
By assumptions, we can write $u$ as a Fourier series
$
u(t) = \sum_{k \in \Z} c_ke^{2\pi i \frac{1}{\beta}k}
$,
where the sequence $(c_k)_k$ is in $\ell^1(\Z)$. Consequently, we have
$$
\ft(\varphi u) = \sum_{k \in \Z} c_k T_{\frac{1}{\beta}k}\varphi \in V_{\frac{1}{\beta}}^1(\varphi).
$$
The statement follows from Theorem \ref{main_theorem2}, since $|\mathcal{G}f(x,\omega)| = |\mathcal{G}\hat f(\omega,-x)|$.
\end{proof}

\subsection{Real-valued maps in Gaussian shift-invariant spaces}

In Section \ref{subse:rv} we have outlined how uniqueness for compactly-supported maps from samples on a 1-dimensional lattice follows under a real-valuedness assumption. This result was a consequence of the fact that a real-valued band-limited map is determined from samples of its absolute value. An extension of the 1-dimensional lattice to a 2-dimensional lattice yielded injectivity for complex-valued maps. In the present section, we show that an analogous statement is valid in the Gaussian shift-invariant setting. The starting point is the following theorem of Gröchenig \cite[Theorem 1]{GrchenigPR} which has the flavor of Thakur's theorem. For $\gamma > 0$ we denote by $\phi_\gamma$ the Gaussian $\phi_\gamma(t) = e^{-\gamma t^2}$.

\begin{theorem}[Gröchenig]\label{gr_phase}
Assume that $\Lambda \subseteq \R$ is separated and that $D^-(\Lambda) > \frac{2}{\beta}$. If $\gamma,\beta>0$ and $f,h \in V_\beta^\infty(\phi_\gamma)$ are such that
$$
|f(\lambda)| = |h(\lambda)| \ \ \forall \lambda \in \Lambda
$$
then there exists a constant $\tau \in \{ 1,-1\}$ such that $f=\tau h$.
\end{theorem}

We notice that Romero generalized Theorem \ref{gr_phase} to the setting where the Gaussian generator is a totally positive function of Gaussian type \cite{Romero2021}. Similar to Section \ref{subse:rv} we can deduce that uniqueness from Gabor measurements under a real-valuedness assumption is achieved from samples on a 1-dimensional lattice. To do so we denote by $V_\beta^p(\varphi, \R)$ the class of real-valued maps in $V_\beta^p(\varphi)$, where $\varphi(t)=e^{-\pi t^2}$ denotes the standard Gaussian.

\begin{corollary}\label{cor:rc}
Let $\varepsilon,\beta>0$ and let $\mathcal{C} = V_\beta^\infty(\varphi, \R)$. Suppose that $\Lambda \subseteq \R$ is separated with lower Beurling density $D^-(\Lambda) > \tfrac{2}{\beta}$. If
$$
X = \Lambda \times \{ 0 \}
$$
then $(\mathcal{C},X)$ satisfies condition \textbf{(U)}.
\end{corollary}
\begin{proof}
Let $n \in \Z$. A calculation shows that the convolution of $\varphi$ with $T_{\beta n} \varphi$ is given by
$$
\varphi * T_{\beta n} \varphi = \tfrac{1}{\sqrt{2}} T_{\beta n} \phi_{\frac{\pi}{2}}.
$$
Hence, if $f = \sum_{n \in \Z} c_n T_{\beta n} \varphi$ with a real sequence $\{ c_n \} \in \ell^\infty(\Z)$ then
\begin{equation}\label{convv}
    \mathcal{G}f(x,0) = \varphi * f = \sum_{n \in \Z} c_n (\varphi *T_{\beta n} \varphi) = \frac{1}{\sqrt{2}} \sum_{n \in \Z} c_n T_{\beta n} \phi_{\frac{\pi}{2}} \in V_\beta^\infty(\phi_{\frac{\pi}{2}}, \R).
\end{equation}
Therefore, the identity
$$
|\mathcal{G}f(z)| = |\mathcal{G}h(z)| \ \ \forall z \in X
$$
holds if and only if
$$
|\varphi * f(x)| = |\varphi * h(x)| \ \ \forall x \in \Lambda.
$$
By equation \eqref{convv}, both convolutions satisfy $\varphi * f, \varphi * h \in V_\beta^\infty(\phi_{\frac{\pi}{2}}, \R)$. Since $D^-(\frac{2}{\beta+\varepsilon} \Z) > \frac{2}{\beta}$ it follows from Gröchenig's theorem that $\varphi * f = \tau (\varphi * h)$ for some $\tau \in \{ 1,-1 \}$. Since the Fourier transform of $\varphi$ does not vanish we have $f=\tau h$, concluding the proof of the statement.
\end{proof}

Comparing the real-valued case with the complex-valued case implies the following statement.

\begin{theorem}\label{thm:rc_c_2}
Let $\varepsilon,\beta>0$. If the signal class $\mathcal{C}_\R$ and the sampling set $X_\R$ is chosen as
$$
\mathcal{C}_\R = V_\beta^1(\varphi,\R), \ X_\R = \frac{2}{\beta+\varepsilon} \Z \times \{ 0 \}
$$
then $(\mathcal{C}_\R,X_\R)$ satisfies condition \textbf{(U)}. Further, if $\beta$ is irrational then the choice 
$$
\mathcal{C}_\C = V_\beta^1(\varphi), \ X_\C = \frac{2}{\beta+\varepsilon} \Z \times \Z
$$
implies that $(\mathcal{C}_\C,X_\C)$ satisfies condition \textbf{(U)}.
\end{theorem}
\begin{proof}
The set $\Lambda = \frac{2}{\beta+\varepsilon} \Z$ has lower Beurling density $D^{-}(\Lambda) > \tfrac{2}{\beta}$ Hence, the first part of the statement follows from Corollary \ref{cor:rc} and the inclusion $V_\beta^1(\varphi,\R) \subseteq V_\beta^\infty(\varphi,\R)$. The second part was shown in Theorem \ref{main_theorem2}.
\end{proof}

\section{Concluding remarks}

We conclude the article with several remarks related to our results from Section \ref{sec:main} and compare them with the previous work mentioned in Section \ref{sec:previous_work}.

\begin{enumerate}
    \item Let $s \in \{-1,1\}$ and $a \in \Z$. In both Theorem \ref{main_theorem} and Theorem \ref{main_theorem2} the lattice $\Z \times (2c)^{-1}\Z$ and $\frac{\beta}{2+\varepsilon}\Z \times \Z$ can be replaced by $(s\N+a) \times (2c)^{-1}\Z$ and $\frac{\beta}{2+\varepsilon}\Z \times (s\N+a)$, respectively. Consequently, compactly-supported functions and functions in shift-invariant spaces with Gaussian generator are determined by spectrogram samples lying in a half-plane.
    \item In the proof of Theorem \ref{main_theorem2} we started by sampling the spectrogram $|\mathcal{G}f(x,\omega)|$ with respect to the time variable $x$. The desired result then followed via an application of Weyl's equidistribution theorem. On the other hand, after applying the sampling Theorem \ref{sampling_si} we could use the identity $\mathcal{G}f(x,\omega) = e^{-2\pi i x \omega} \mathcal{G}\hat f (\omega,-x)$ and then apply the Fourier transform with respect to the second argument as we did in Theorem \ref{main_theorem}, the case of compactly-supported signals. This approach would yield the identity
    $$
    \ft |\mathcal{G}\hat f (\omega,\cdot)|^2(x) = \langle (\hat f)_x,T_\omega \varphi_x \rangle.
    $$
    Since $f$ is an element of $V_\beta^1(\varphi)$, the function $(\hat f)_x$ is certainly not compactly-supported and Zalik's theorem, Theorem \ref{thm:zalik}, is not applicable. At this point, one could use the following generalization of Zalik's theorem \cite[Theorem 2]{Zalik}.
    \begin{theorem}
    Let $\psi(t)=ae^{-b^2(t-c)^2}$ be a Gaussian where $a,b>0$ and $c \in \R$. Let $\Lambda \subseteq \R$ be a sequence of distinct real numbers and denote by $S(\varepsilon)$ the series
    $$
    \sum_{\lambda \in \Lambda}^{}{}^{'} \frac{1}{|\lambda|^{2+\varepsilon}}.
    $$
    The divergence of $S(\varepsilon)$ for some $\varepsilon>0$ suffices for the sequence $\{ T_{\lambda}\psi : \lambda \in \Lambda \}$ to be complete in $\lt$, whereas the divergence of $S(0)$ constitutes a necessary condition for completeness of $\{ T_{\lambda}\psi : \lambda \in \Lambda \}$ in $\lt$.
    \end{theorem}
    However, an application of this result is not meaningful for our purposes. For if $\Omega$ is separated then $S(\varepsilon)$ certainly converges for every $\varepsilon>0$. Therefore, this approach would not yield sampling sets consisting of separated points.
    \item At the beginning of the present paper it was shown that if $e^{i\alpha_1}\R, e^{i\alpha_2}\R \subseteq \C$ are two lines which satisfy the condition $\alpha_1 - \alpha_2 \notin \pi\Q$ then the pair $(\lt, e^{i\alpha_1}\R \cup e^{i\alpha_2}\R)$ satisfies condition \textbf{(U)}. This was a consequence of Jaming's theorem \cite{Jaming}. A crucial step in the proof of Jaming's theorem lies in the observation that if $f$ is a non-zero entire function then for arbitrary $z \in \C \setminus \{ 0 \}$ the orbit $D = \{ ze^{2i(\alpha_1-\alpha_2)n} : n \in \Z \}$ cannot be contained in the zero set of $f$. For if $\alpha_1 - \alpha_2 \notin \pi\Q$ then by Weyl's equidistribution theorem the set $D$ contains a limit point. In Theorem \ref{main_theorem2}, Weyl's equidistribution theorem was used to obtain uniqueness results from fully discrete measurements.
\end{enumerate}

\vspace{0.5cm}

\noindent\textbf{Acknowledgements.} The authors would like to thank Karlheinz Gröchenig for helpful remarks concerning Section \ref{sec:basicresults}.

\bibliography{bibfile}{}
\bibliographystyle{acm}

\end{document}